\documentclass{article}
\usepackage{amsmath,amsfonts,amsthm,amscd,amssymb,latexsym}
\usepackage{enumerate,comment}
\usepackage{graphicx,psfrag}
\newtheorem*{theorem}{Theorem}
\newtheorem*{prop}{Proposition}
\theoremstyle{definition}
\newtheorem{defn}{Definition}
\newtheorem*{exx}{Example}
\newtheorem*{cor}{Corollary}

\newcommand{\FF}{\ensuremath{\mathbb{F}}}

\renewcommand{\a}{\alpha}
\renewcommand{\b}{\beta}
\renewcommand{\c}{\gamma}

\newcommand{\G}{\Gamma}
\newcommand{\pprime}{{\prime\prime}}
\newcommand{\bs}{\backslash}
\newcommand{\be}{\begin{enumerate}}
\newcommand{\ee}{\end{enumerate}}

\title{Square free words as products of commutators.}

\author{\sc Andrew Duncan, Alina Vdovina}
\title{Square free words as products of commutators.
}

\date{\today}
\begin{document}

\maketitle

\begin{abstract}
Elements of the commutator subgroup of a free group $\FF$ can be presented
as values of canonical forms, called Wicks forms. 
We show that, starting from sufficiently high genus $g$, there is a 
sequence of words $w_g$ which can be presented by $f(g)$ distinct Wicks forms, 
where $f(g)>g!$. Moreover we may choose  these
words $w_g$ to be 
square free.  
\end{abstract}

\section*{Introduction}

Let $\FF$ be a free group and let $[\FF,\FF]$ be its commutator subgroup.
We define the {\em genus}  of a word
 $w \in [\FF,\FF]$ to be the least positive integer $g$ such that $w$ 
is a product
 of $g$ commutators in $\FF$. 
Every element of genus $g$ in  $[\FF,\FF]$ can be presented
by non-cancelling substitution in a Wicks form of genus $g$
 ([cf. \cite{C,CE}]) 
as described 
below.  In genus one there is only one Wicks form but for $g \geq 2$ 
there are finitely many Wicks forms  and their number grows quite fast (factorially)
with $g$.
 Thus a  natural question (posed by E.Rips at 
the geometric group theory conference in Anogia in 1996) is  
whether it is possible to find a word of 
genus $g$ which can be presented by non-cancelling substition
in ``many'' Wicks forms.
We show that there is
a sequence of words $w_{2}, w_{3},\ldots $, such that $w_g$ has
genus $g$ and the number of Wicks forms from which it can be
obtained by non-cancelling substitution 
is bounded below by $g!$, when $g$ is sufficiently large.

  Wicks forms are not affected by Dehn twists.
Since, by a classical result of M.~Dehn (\cite{D}), the modular
group is generated by Dehn twists, Wicks forms are invariant
under the action of $Mod(S_g)$. In 
\cite{BF} Bestvina and Feighn consider genus $g$ 
representations (defined below) of words
of genus $g$. They  describe an equivalence relation on representations
in terms of partial Dehn twists, which are a generalisation of Dehn twists. 
They show that for
certain types of word $w_g$ the number of distinct equivalence classes 
of representations
of $w_g$ grows exponetially with $g$.  
It 
would be interesting to know what happens to our Wicks form representations
under fractional Dehn twists.

\section{Wicks forms}
An {\em alphabet} consists of  a countably enumerable set $S$ 
equipped with
a fixed-point free involution $\tau$ and a fixed set of representatives $S^+$ of the
orbits of $\tau$. 
We shall consider words in 
a fixed alphabet  $A=\{a_1^{\pm 1},a_2^{\pm 1},\ldots\}$
of letters
 $a_1,a_2,\ldots$ and their inverses $a_1^{-1},a_2^{-1},\ldots$, where
$\tau(a_i)=a_i^{-1}$ and $\tau(a_i^{-1})=a_i$ and  
 $A^+=\{a_1,a_2, \ldots \}$. 
 Alphabets $B$, $A^\prime$ are defined analagously, replacing $a_i$ by $b_i$ or $a_i^\prime$, in the obvious way. 
A {\em word} 
({\em over } $A$) is an
element of the free monoid $A^*$ on $A$. A {\em cyclic word} ({\em over } $A$)
is the set
of cyclic permutations $[w]$  of a word $w$. 
Words $w$ and $w^\prime$ are said to {\em define the same cyclic word}
if $[w]=[w^\prime]$.
A word $u$ is said to be a {\em factor}
of the cyclic word  $[w]$ if $u$ is a subword of some element of $[w]$. A
word is said to be {\em reduced} if it has no factor of the form $aa^{-1}$ or
$a^{-1}a$, where $a\in A^+$. A word is said to be {\em cyclically reduced} if every element of 
$[w]$ is reduced. An element $a^\epsilon$, where $a\in A^+$ and $\epsilon=\pm 1$, is
said to {\em occur} $n${\em -times} in the word $w$ if 
$w=u_1 a^\epsilon u_2 \cdots u_n a^\epsilon u_{n+1}$, for some elements
$u_1,\ldots ,u_{n+1}$ of $(A\backslash\{a^\epsilon\})^*$: we count the 
number of occurrences of $a^{+1}$ and $a^{-1}$ separately.

\begin{defn}
 An {\em orientable Wicks form\/} is a cyclic word $[w]$ over $A$ 
such that
\be[(i)]
\item\label{wi} if $a^\epsilon\in A$ occurs in $w$ (for $a\in A^+$ and 
$\epsilon\in\{\pm 1\}$)
then $a^{-\epsilon}$ occurs exactly once in $w$;
\item\label{wii}  $w$ is cyclically reduced and 
\item\label{wiii} if $a_i^\epsilon a_j^\delta$ is a factor of $[w]$ then
 $a_j^{-\delta}a_i^{-\epsilon}$ is not a factor of $[w]$.
\ee
\end{defn}
We shall abuse notation by referring to a Wicks form $[w]$ as $w$, when convenient.
 An orientable Wicks form $[w]=[w_1w_2\cdots]$, over $A$, is
 {\em isomorphic\/} to $[w']=[w'_1w'_2\ldots ]$, over the  alphabet $A'$, if
 there exists a bijection $\varphi:A\longrightarrow A'$ 
with $\varphi(a^{-1})=\varphi(a)^{-1}$ such that 
$[\varphi(w)]=[w']$ (where
 $\varphi(w)=\varphi(w_1)\varphi(w_2)\cdots$). The relation 
``is isomorphic to'' is an equivalence relation on the orientable
Wicks forms over $A$.

 If $[w]$ is an orientable Wicks form then 
 $w$, when considered as an element of the free group $\FF$ generated by
 $a_1,a_2,\ldots$, is an element of the commutator subgroup.
 We define the {\em algebraic genus\/} $g_a(w)$ of
 $[w]$ to be the least positive integer $g_a$ such that $w$ is a product
 of $g_a$ commutators in $\FF$.

 The {\em topological genus\/} $g_t(w)$ of an orientable Wicks
 form $w=w_1\cdots w_{2e-1}w_{2e}$ 
is defined to be the topological
 genus of the orientable compact connected surface obtained by
 labelling and orienting the edges of a $2e-$gon (which we
 consider as a subset of the oriented plane) according to
 $w$ and then identifying edges with the same labels (respecting orientation).

\begin{prop}[cf. \cite{C,CE}]
The algebraic genus and the topological genus of an orientable Wicks
 form coincide.
\end{prop}

 We define the {\em genus\/} $g(w)$ of an orientable
 Wicks form $[w]$ to be $g(w)=g_a(w)=g_t(w)$.

 Consider the orientable compact surface $S$ associated to an orientable 
 Wicks form $w=w_1\cdots w_{2e}$. This surface carries an embedded graph
 $\Gamma\subset S$ such that $S\setminus \Gamma$ is an open polygon
 with $2e$ sides (and hence connected and simply connected).
 Moreover, conditions (\ref{wii}) and (\ref{wiii}) on the Wicks form imply that $\Gamma$ 
 contains no vertices of degree $1$ or $2$ (or equivalently that the
 dual graph of $\Gamma\subset S$ contains no faces which are $1-$gons
 or $2-$gons). This construction also works 
 in the opposite direction: given a graph $\Gamma\subset S$
 with $e$ edges and no vertices of degree $1$ or $2$ on an orientable compact connected surface $S$ of genus $g$
 such that $S\setminus \Gamma$ is connected and simply connected, we get
 an orientable Wicks form of genus $g$ and length $2e$ by labelling and 
 orienting the edges of $\Gamma$ and by cutting $S$ open along the graph
 $\Gamma$. The associated orientable Wicks form is defined as the word
 which appears in this way on the boundary of the resulting polygon
 with $2e$ sides.  Henceforth we identify orientable Wicks
 forms with their associated embedded graphs $\Gamma\subset S$,
 speaking of vertices and edges of orientable Wicks forms.
 
 The formula for the Euler characteristic
 $$\chi(S)=2-2g=v-e+1$$
 (where $v$ denotes the number of vertices and $e$ the number
 of edges in $\Gamma\subset S$) shows that
 an orientable Wicks form of genus $g$ has  length at least $4g$
 (the associated graph has then a unique vertex of degree $4g$
 and $2g$ edges) and   length at most $6(2g-1)$ (the associated
 graph has then $2(2g-1)$ vertices of degree three and
 $3(2g-1)$ edges).

 We call an orientable Wicks form of genus $g$ {\em maximal\/} if it has
 length $6(2g-1)$.
It is convenient to interpret a genus $g$ Wicks form
as an {\it oriented circuit} in the graph $\Gamma$; that is
a circuit which traverses every edge of the graph exactly
twice, once in each direction, but contains no instance of an
edge followed by its reverse
(for more details concerning oriented circuits
see \cite{V}). The dual graph of Wicks form is an ideal triangulation,
as defined by L.Mosher \cite{M}. It is shown in \cite{M} 
that ideal triangulations are invariant under Dehn twists.

Let $g$ be a positive integer,  $A=\{a_1^{\pm 1},\ldots,a_{6g-3}^{\pm 1}\}$ 
and 
$B=\{b_1^{\pm 1},\ldots ,b_{n}^{\pm 1} \}$ 
be  alphabets and let $A^*$ and $B^*$ be the free monoids on 
$A$ and $B$, respectively. 
Define $(a_i^{-1})^{-1}=a_i$ and  
for a word $w=x_1\ldots x_n\in A^*$ with $x_i\in A$ define 
$w^{-1}=x_n^{-1}\ldots x_1^{-1}$.  
Let $\Phi$ be a map from $A$ to $B^*$, such that 
$\Phi(a_i)$ is freely reduced and 
$\Phi(a_i^{-1})=\Phi(a_i)^{-1}$, for all $i$. Then $\Phi$ induces
a map, also denoted $\Phi$, from $A^*$ to $B^*$. 
Let $u\in A^*$ and $w$ be the word obtained from 
$\Phi(u)$ by free reduction. 
Then we say that $(\Phi,u)$ is a {\em representation} of $w$.
If $\Phi(u)$  is cyclically reduced
then $(\Phi,u)$ is said to be a 
\emph{non-cancelling representation} of $w$. 
 
Given a Wicks form $[u]$  of genus $g$   
we say that
$(\Phi, u)$ is a {\em genus} $g$ {\em Wicks representation} 
of $w=\Phi(u)$ if there
is $u^\prime\in [u]$ such that $(\Phi,u^\prime)$ is a representation of $w$.
If $u$ and $v$ are isomorphic genus $g$  Wicks forms then there exists
a genus $g$ Wicks representation $(\Phi,u)$ of $w$ if and only if there exists 
a  genus $g$ Wicks representation $(\Psi,v)$ of $w$.  
\begin{defn}\label{def:mgw}
For $w\in B^*$ we define $M(g,w)$ to be the number of isomorphism 
classes $I$ of maximal 
Wicks forms over $A$, of genus $g$, such that there exists
a non-cancelling Wicks representation $(\Phi,u)$ of $w$, for some $u\in I$. 
\end{defn}

We shall say a word $w\in B^*$ has {\em genus} $g$ if $w$  
has a non-cancelling,  genus $g$, Wicks representation  
but no 
Wicks representation of genus less than $g$. 
If $w$ does not have a non-cancelling 
Wicks representation, of any genus,  we define its genus to be $\infty$. 
A cyclically reduced word $w$ has
genus $g<\infty$ if and only if it represents an element of the commutator subgroup
of  the free group on $\{b_1,\ldots, b_n\}$ (see  \cite{CE}).

\section{Words with Many Representations}
In the sequel, if $u\in A^*$ and $w\in B^*$, for 
alphabets $A$ and $B$, and there exists a non-cancelling representation
$(\Phi,u)$ of $w$ then we shall say that $w$ {\em is obtained from }
 $u$ {\em by non-cancelling substitution}. 
\begin{theorem}\label{thm:main}
There is a sequence $w_2,w_3,w_4,\ldots $ of words over an alphabet $B$ 
of size $24$ 
 such that 
\be[(i)]
\item $w_g$ has genus $g$ and 
\item  $M(g,w_g)>g!$, for $g>10^{10}$. 
\ee
\end{theorem}
\begin{proof}
Consider a maximal Wicks form $w$ of genus $g$ and length $12g-6$.
We colour vertices of the graph $\Gamma$ of $w$ so that no edge
is incident to two vertices of the same colour. A straightforward induction 
on the number of vertices 
shows that, for a simple graph  whose vertices have degree at most $d$, vertices
may be coloured in this way using at most  
$d+1$ colours. In our case 
$\G$ has no vertex of degree more than $3$ (in fact is regular of degree $3$) so four
colours $1,2,3,4$ are sufficient.

We construct a new labelled oriented graph $\G^\prime$ on the
surface $S$ associated to $w$, as follows. The underlying graph of $\G^\prime$ is
the barycentric subdivison of $\G$. Thus $\G^\prime$ has two types of 
vertex; the original vertices of $\G$ which are of degree $3$, and new vertices
of degree $2$ corresponding to edges of $G$. All edges of $G^\prime$ join a vertex
of degree $3$ to a vertex of degree $2$ and we orient each edge from the degree $3$
vertex to the degree $2$ vertex. The vertices of $\G^\prime$ degree $3$ 
inherit a colouring from the corresponding vertices of $\G$.   
Now let $B=\{\a_j^{\pm 1},\b_j^{\pm 1},\c_j^{\pm 1}|j=1,2,3,4\}$ be an alphabet
disjoint from $A$. 
Let $u$ be a vertex 
of $\G^\prime$ of degree $3$ and colour $j$. 
Choose one of the $3$ edges emanating from $u$ and call it $e_1$. Choose one of the
other edges and call it $e_2$ and call the third $e_3$. 
Label the oriented edges $e_1, e_2$ and $e_3$ coming out of $u$ with $\a_j$, $\b_j$ by $\c_j$,
respectively, as shown in Figure \ref{fig:label}. 
\begin{figure}
  \begin{center}
    \psfrag{aj}{$\a_j$}
    \psfrag{bj}{$\b_j$}
    \psfrag{cj}{$\c_j$}
    \includegraphics[scale=0.4]{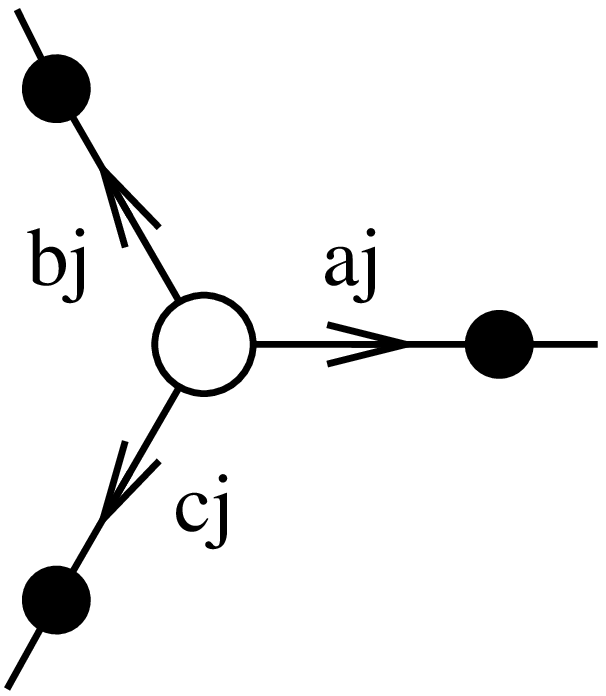}
  \end{center}
  \caption{Labelling of edges of  $\G^\prime$}\label{fig:label}
\end{figure}
Repeat this for all degree $3$ vertices of $\G^\prime$.
Now reading the oriented 
circuit $C$ with this labelling we obtain a word 
$v \in B^*$. 
Moreover, the condition on the
colouring of vertices of $\G$ means that 
$v$ is obtained from $w$ by
non-cancelling substitution. 

To simplify the estimation of the lower bound below 
we now make a further adjustment to the current 
labelling of $\G^\prime$. Let $u$ be a vertex of degree $3$ 
with outgoing edges labelled 
$\a_j,\b_j,\c_j \in B$. Then either $\a_j^{-1}\b_j$ or $\b_j^{-1}\a_j$ 
is a factor of $w^\pprime$.
Moreover, if $ \a_j^{-1}\b_j$ is a factor of $w^\pprime$ then so are
$\b_j^{-1}\c_j$ and $\c_j^{-1}\a_j$, whereas if $\b_j^{-1}\a_j$ 
is a factor then
so are $\a_j^{-1}\c_j$ and $\c_j^{-1}\b_j$. 
In the latter case we alter the labelling
on the edges incident to $u$ by interchanging labels $\c_j$ and $\b_j$: as 
shown in Figure \ref{fig:adjust}, where the curved arcs indicate the direction
in which the boundary of $S\backslash \G^\prime$ is read.
\begin{figure}
  \begin{center}
    \psfrag{aj}{$\a_j$}
    \psfrag{bj}{$\b_j$}
    \psfrag{cj}{$\c_j$}
    \includegraphics[scale=0.4]{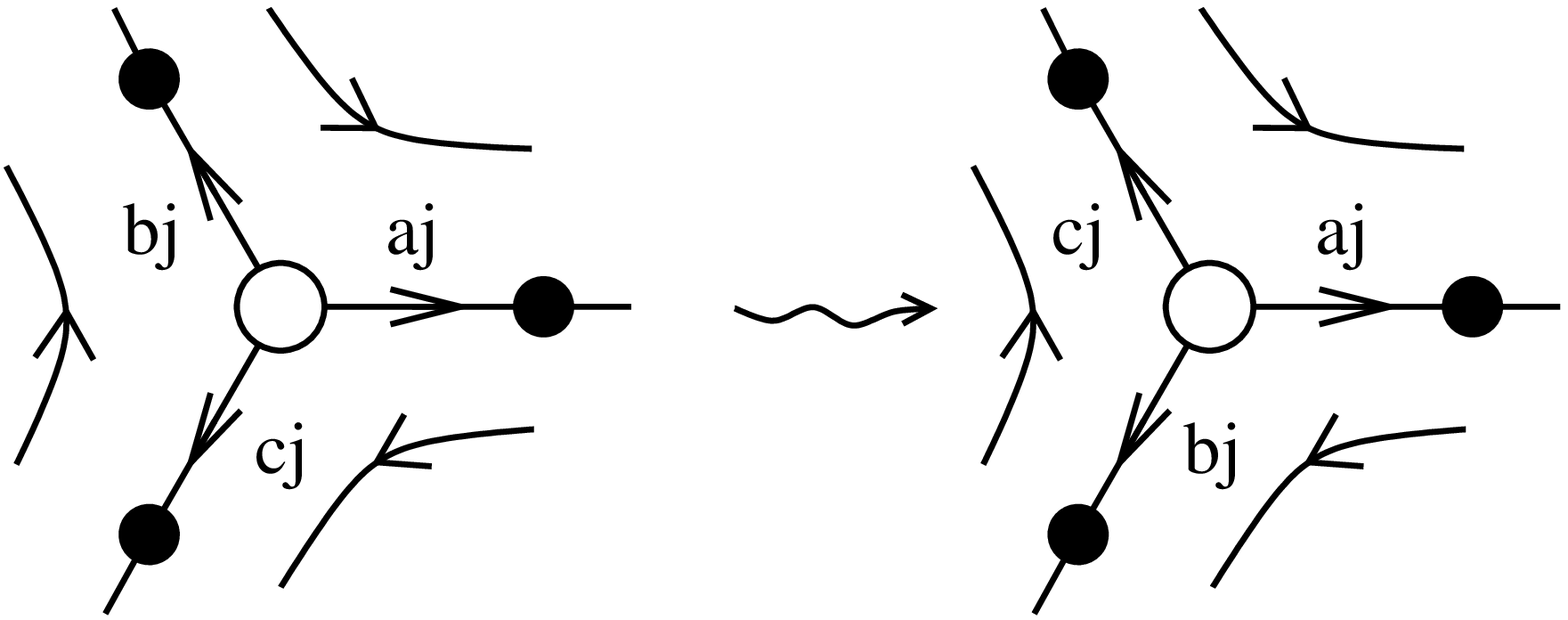}
  \end{center}
  \caption{Adjustment of labels on $\G^\prime$}\label{fig:adjust}
\end{figure} 
The result is
that $\a_j^{-1}\b_j$, $\b_j^{-1}\c_j$ and  $\c_j^{-1}\a_j$ 
are subwords of the (altered) word
$w^\pprime$. We perform this alteration, 
if necessary, on every vertex, of degree $3$, of $\G^\prime$.  
By reading the circuit $C$ with these altered labels we obtain a new 
word $v$; also
obtained from $w$ by non-cancelling substitution. Furthermore this word $v\in B^*$ 
satisfies the
following  conditions.
\be[(i)]
\item\label{it:v1} Every letter of $v$ from $B^+$ of colour $j$ is  followed by a letter of 
 colour $k\neq j$ from $B\bs B^+$. 
\item\label{it:v2}  Every letter of $v$ from $B\bs B^+$ is followed by a letter of the same colour
from $B^+$.
\item\label{it:v3} 
None of 
$\b_j^{-1}\a_j$, 
$\a_j^{-1}\c_j$ or $\c_j^{-1}\b_j$ occur as subwords of $v$.
\item\label{it:v4} The length of $v$ is $24g-12$.
\ee
An example of this construction is given after the proof. 

Let's show that $v$ cannot be obtained by a non-cancelling
substitution from a Wicks form of genus $k$, where $k<g$.
Suppose that $(\Phi,u)$ is a genus $k$ non-cancelling 
representation of $v$. Since a Wicks form of genus $k$ has length at most
$12k-6\le 12(g-1)-6$ and $v$ has length $24g-12$ there must be
some letter $a$ occuring in $u$ such that $\Phi(a)>2$.  
Let $a$ be such a letter. Then $\Phi(a)$ contains a subword $xyz$, where
$x,y,z\in B$. As $a^{-1}$ also occurs in $u$ both $xyz$ and  $z^{-1}y^{-1}x^{-1}$ 
must be subwords of $v$. 
However, from \eqref{it:v1} and \eqref{it:v2} above,
 positive and negative powers of letters of $v$ alternate.
Moreover, if $j$ is fixed and $x,y\in \{\a_j,\b_j,\c_j\}$ then   
$v$ may have subword $x^{-1}y$ but does not have subword $xy^{-1}$.
This means that a subword of $v$ of length three always has  form
$x_i^{-1}y_iz_j^{-1}$ or $x_iy_j^{-1}z_j$, where $x_i^{\pm 1},
y_i^{\pm 1},y_j^{\pm 1}$ and $z_j^{\pm 1}$ are elements of $B$.
Therefore $v$  cannot
contain both subwords $xyz$ and $z^{-1}y^{-1}x^{-1}$, and the 
result follows. Hence the genus of $v$ is $g$. 

Now we wish to count the number of words $v$ that may arise in this way. First note
that 
we may  read the word $v$ starting from any vertex of $\G^\prime$ and
so we may choose $v$ so that it begins with $\a_j$, for some $j\in \{1,2,3,4\}$.
 Let $V(g)$ denote the set of all freely reduced words of $B^*$ which satisfy
 \eqref{it:v1} to \eqref{it:v4} above and which begin with $\a_j$, for some $j$;
 so $v\in V(g)$. 
If $u$ belongs to $V(g)$ and $x\in \{\a_s, \b_s,\c_s\}$ occurs in $u$ then the
letter following $x$ belongs to $\cup_{i\neq s}\{\a_i^{-1},\b_i^{-1},\c_i^{-1}\}$;
so there are nine possibilities for this letter. If $y$ is a letter of $B\bs B^+$ 
occuring in $u$ then the letter following $y$ is completely determined by conditions
\eqref{it:v2} and \eqref{it:v3} above. Given that there are $4$ choices for the
first letter $\a_j$ of $u$ and that $u$ has length $24g-12$ this means that 
$|V(g)|=4\cdot 9^{12g-7}$.  
 
Now let $M(g)=\max\{M(g,v)|v
\textrm{ is a reduced word of genus } g \textrm{ and length } 24g-12\}$ and 
let $W(g)$ be the number of isomorphism classes of maximal oriented Wicks
forms of genus $g$. Then $|V(g)|M(g)\ge W(g)$. 
However it follows from \cite{BV} that 
$W(g)$ grows faster than $g!$; in fact 
setting 
\[m(g)=\left(\frac{1}{12}\right)^g\frac{(6g-4)!}{g!(3g-2)}\]
we have $W(g)\ge m(g)$. (In \cite{BV} $m(g)=m_1^g$.)
Using upper and lower bounds for $n!$, from \cite{R},
straightforward calculations give
\[
M(g)\ge W(g)/M(g)\ge m(g)/(4\cdot 9^{12g-7})> g!,
\textrm{ for }g>10^{10},
\]
and 
the statement of the theorem follows. 
\end{proof}
\begin{exx}
The quadratic word 
\[w=a_1a_2a_3a_4a_5a_1^{-1}a_6a_2^{-1}a_5^{-1}a_7a_8a_3^{-1}a_6^{-1}a_9a_7^{-1}
a_4^{-1}a_8^{-1}a_9^{-1}\] gives rise to the labelled graph  of Figure \ref{fig:ex1},
embedded on an orientable surface of genus $2$; so this word is a maximal Wicks
form of genus $2$. (The small curved arrows indicate the direction to be taken
at each vertex in reading the circuit $C$.)  
A colouring of the graph with colours $1$, $2$ and $3$ is
also shown in the figure. 
\begin{figure}
  \begin{center}
    \psfrag{a1}{$a_1$}
    \psfrag{a2}{$a_2$}
    \psfrag{a3}{$a_3$}
    \psfrag{a4}{$a_4$}
    \psfrag{a5}{$a_5$}
    \psfrag{a6}{$a_6$}
    \psfrag{a7}{$a_7$}
    \psfrag{a8}{$a_8$}
    \psfrag{a9}{$a_9$}
    \psfrag{1}{$1$}
    \psfrag{2}{$2$}
    \psfrag{3}{$3$}
     \includegraphics[scale=0.4]{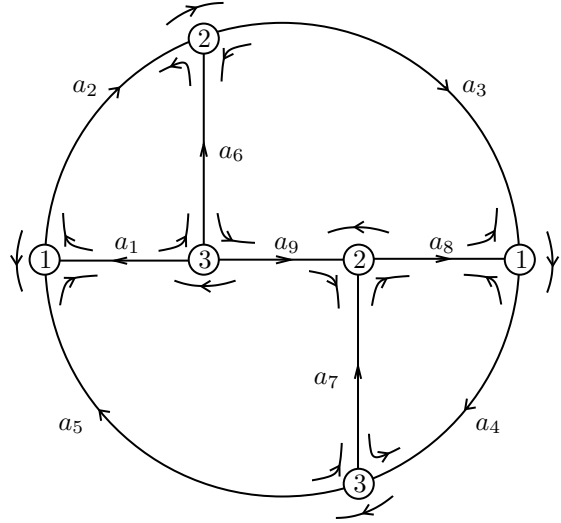}
  \end{center}
  \caption{A maximal Wicks form}\label{fig:ex1}
\end{figure} 
The barycentric subdivision of this graph with its new orientation and
a choice of new labelling is shown in Figure \ref{fig:ex2}. 
\begin{figure}
  \begin{center}
    \psfrag{a1}{$\a_1$}
    \psfrag{a2}{$\a_2$}
    \psfrag{a3}{$\a_3$}
    \psfrag{b1}{$\b_1$}
    \psfrag{b2}{$\b_2$}
    \psfrag{b3}{$\b_3$}
    \psfrag{c1}{$\c_1$}
    \psfrag{c2}{$\c_2$}
    \psfrag{c3}{$\c_3$}
    \psfrag{1}{$1$}
    \psfrag{2}{$2$}
    \psfrag{3}{$3$}
     \includegraphics[scale=0.4]{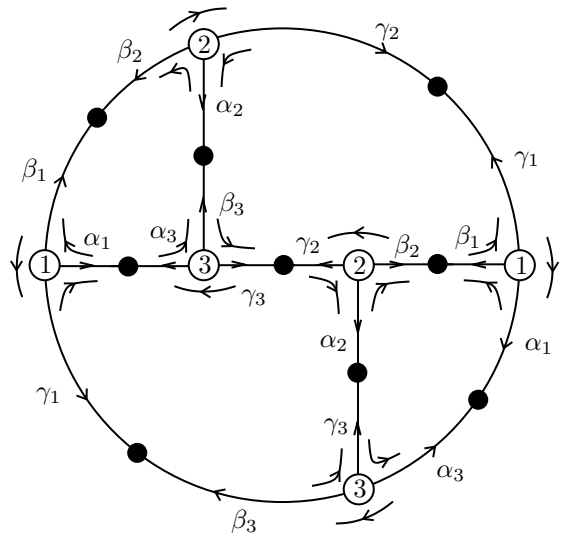}
  \end{center}
  \caption{Barycentric subdivision and new labelling}\label{fig:ex2}
\end{figure} 
Starting by reading the edge labelled $\a_1$ coming out of the rightmost
vertex we obtain the word
\begin{align*}
v=& \a_1\a_3^{-1}\b_3\c_1^{-1}\a_1\a_3^{-1}\b_3\a_2^{-1}\b_2\b_1^{-1}\c_1\b_3^{-1}\c_3
\a_2^{-1}\b_2\b_1^{-1}\c_1\c_2^{-1}\\
 &\a_2\b_3^{-1}\c_3\c_2^{-1}\a_2\c_3^{-1}\a_3\a_1^{-1}\b_1\b_2^{-1}
\c_2\c_3^{-1}\a_3\a_1^{-1}\b_1\b_2^{-1}\c_2\c_1^{-1}.
\end{align*}
No adjustment of the labels is required as $v$ contains none of the  forbidden 
subwords. 
\end{exx}

As mentioned in the introduction the objects counted in \cite{BV} are
equivalence classes of representations of particular words $w_g$; whereas
here we count (isomorphism classes of) representations themselves. It seems
plausible that by choosing the words $w_g$ carefully the size of equivalence
classes could be made small, but we have no proof of any such claim. We can 
however offer some evidence that it is possible, by making a moderate increase in 
the  size of the alphabet, to choose the words $w_g$ from 
a much more limited set than $V(g)$ above. 
A word in $B^*$ is said to be {\em square free} if it contains no
subword of the form $uu$, where $u$ is a non-trivial element of $B^*$.
Such square free words exist (see for example \cite[Chapter 4]{L}).
\begin{cor}
In the theorem above, replacing $B$ with an alphabet of size $80$, 
the words $w_g$ may be chosen to be square free
 and the conclusion of the theorem then holds for $g>10^{17}$.
\end{cor}
\begin{proof}
Let $w$ of be a maximal Wicks form of  genus $g$ and length $12g-6$ and,
 as in the proof of the theorem, let $v$ be the label of the oriented circuit
$C$ on the surface corresponding to $w$ with the labelling constructed in the
proof. 
Now for each letter $x_j\in \{\a_j,\b_j,\c_j\}\subseteq B$,   where 
$j=1,2,3$ or $4$, 
define a set $E^+(x_j)=\{x_{j,1},x_{j,2},x_{j,3}\}$ and 
define $E(x_j)=E^+(x_j)\cup \{x^{-1}|x\in E^+(x_j)\}$ and 
$E=\cup_{j=1}^4 (E(\a_j)\cup E(\b_j)\cup E(\c_j))$.
Further let 
$E_S=\{x_{1,0}^{\pm 1}, x_{2,0}^{\pm 1}, x_{3,0}^{\pm 1},x_{4,0}^{\pm 1}\}$ and 
let $\hat E=E\cup E_S$.

Fix some element $x_i^{+1}\in B$ which occurs in $v$.
If $x_i$ occurs $s$ times in $v$,
then $v$ can be written as $v=u_1x_iu_2x_iu_3\ldots x_iu_sx_iu_{s+1}$
where the words $u_1,u_2, \ldots,u_s,u_{s+1}$ don't contain
$x_i$ (they contain $x_i^{-1}$, of course, but we are interested
in positive powers of $x_i$ only).
 Now $v$ is transformed as follows. Choose a square free word 
$t_1\cdots t_{s-1}$ of length $s-1$ in $E(x_i)^*$. 
Replace the first occurrence of $x_i$ with $x_{i,0}$ and replace
the $j$th occurrence of $x_i$ with $t_{j-1}$, for $j=2,\ldots, s$.
Reading the oriented circuit $C$ each occurrence of $x_i$ is paired to an
occurrence of $x_i^{-1}$ in $v$; as each edge is read in both directions. 
The occurrence $x_i^{-1}$ paired to the $j$th occurrence of $x_i$ is now replaced
by $x_{i,0}^{-1}$, if $j=1$, and by $t_{j-1}^{-1}$, otherwise.
(The special letter $x_{i,0}$ is introduced to ensure that all cyclic
permutations of the resulting word are 
square free, as we are dealing with cyclic words.)
Repeat this transformation with respect to every
letter of positive exponent occurring in  $v$
and denote
the resulting word by $z$. Then $z\in \hat E^*$ is obtained by non-cancelling substitution
from $w$.

It is easy to see that $z$ is square free as follows.
Suppose not: then $z$ has a subword $uu$.  Let $u=r_1r_2...r_p$, where $r_i\in E$.
Set $j_1=1$ if $r_1$ 
is an element of $E$ of exponent $+1$ and set $j_1=2$ otherwise.
Since $z$ consists of alternating
positive and negative powers then $r_{j_1}$ will be a positive power.
Then $r_{j_1}=e_i\in E(x_i)$, for some $i$. Suppose that there are $p$ occurrences
of $e_i$ in $u$, namely $r_{j_1},\ldots, r_{j_p}$, with $j_1<\cdots <j_p$.
Since $uu$ is a subword of $z$ the word $r_{j_1}\cdots r_{j_p}r_{j_1}\cdots r_{j_p}$
is a subword of the word $t=t_1\cdots t_{s-1}$ used to construct $z$, above. This 
condradicts the choice of $t$ as a square free word. 
Hence $z$ is square free as claimed.

The argument of the proof of the theorem shows that $z$ has genus $g$. 
As $v$ begins with $\a_j$ the first letter of $z$ is $x_{j,0}$ and so there 
are $4$ possibilities for this letter. Each other letter of $z$ of positive
exponent is obtained by substition of one of the three elements of $E(x_j)$ for
$x_j$, for some $x_j\in B^+$; or is the first letter of its kind in $z$ in which
case it is a uniquely determined element of $E_S$. It follows that there are at most 
$4\cdot 27^{12g-7}$ possibilities for $z$.   
The result now follows. 
\end{proof}

\noindent{\bf Acknowledgements} The authors are indebted to 
A.G.~Miasnikov for his useful 
comments and suggestions on this work.

{\small
{Andrew Duncan}\\
email: a.duncan@ncl.ac.uk\\[1em]
{Alina Vdovina}\\
email: {Alina.Vdovina@newcastle.ac.uk}\\[1em]
{School of Mathematics and Statistics, Newcastle University,
Newcastle upon Tyne, NE1 7RU, United Kingdom}
}
\end{document}